\documentclass[a4paper,12pt]{article}
\usepackage{cmap}                        
\usepackage[cp1251]{inputenc}            
\usepackage[english]{babel}
\usepackage[left=2cm,right=2cm,top=2cm,bottom=2cm]{geometry} 

\usepackage{amssymb}
\usepackage{amsmath, amsthm}
\theoremstyle{plain}
\newtheorem{thm}{Theorem}
\newtheorem{lem}{Lemma}

\newtheorem{cor}{Corollary}
\newtheorem{conj}{Conjecture}

\theoremstyle{definition}

\newtheorem{defn}{Definition}

\begin{document}

\begin{center}\Large
\textbf{On a generalization of the concept of $S$-permutable subgroup of a finite group}\normalsize

\smallskip
V.\,I. Murashka

 \{mvimath@yandex.ru\}

 Francisk Skorina Gomel State University, Gomel\end{center}

\textbf{Abstract.} Let $\sigma=\{ \pi_i | i\in I$ and $\pi_i\cap\pi_j=\emptyset$ for all $i\neq j\}$ be a partition  of the set of all primes   into mutually disjoint  subsets. In this paper we considered subgroups that permutes with  given sets of $\pi_i$-maximal subgroups for all $\pi_i\in \sigma$. In particular we showed that such subgroups forms a sublattice of the lattice of all subgroups of a finite group. As corollaries we obtained some well known results about $S$-permutable subgroups.



\textbf{AMS}(2010). 20D20,  20D30,   20D40.

\section*{Introduction and the results}

All considered groups are finite. All through this paper we denote by    $\sigma=\{ \pi_i | i\in I$ and $\pi_i\cap\pi_j=\emptyset$ for all $i\neq j\}$ some partition  of the set of all primes   $\mathbb{P}$ into mutually disjoint  subsets.

 A subgroup $H$ of a group $G$ is said to permute with a subgroup $K$ if $HK$ is a subgroup of $G$.   $H$ is said to be   $S$-permutable in $G$ if it permutes with every Sylow subgroup of $G$. According to O.\,H. Kegel \cite{sp1} and  W.\,E. Deskins \cite{sp2} if $H$ is $S$-permutable in $G$ then $H^G/H_G$ is nilpotent. Moreover, the set of all $S$-permutable subgroups of $G$  forms a sublattice of the lattice of all   subgroups of $G$ (see \cite[Statz 2]{sp1}). P. Schmid \cite{sp3} showed that if $H$ is $S$-permutable subgroup of a group $G$ then $N_G(H)$ is also $S$-permutable in $G$.

 The concepts of $S$-permutable subgroup plays important role in the structural study of finite non-simple groups.   That is why there were several attempts to generalize this concept.   A.\,N. Skiba  \cite{sp4} suggested the following generalization of the concept of $S$-permutable subgroup. A subgroup $H$ of a group $G$ is called $\sigma$-permutable if $G$ has a Hall $\pi_i$-subgroup $P_i$ with $HP_i^x=P_i^xH$ for every $x\in G$  and  every $\pi_i\in\sigma$.   In particular, A.\,N. Skiba obtained the analogues of the results of O.\,H. Kegel and E.\,W. Deskins for groups with some complete sets of Hall subgroups.

 The main disadvantage of the concept of $\sigma$-permutable subgroup is that it requires the existence  of Hall subgroups. The aim of this paper is to extend the theory of $\sigma$-permutable subgroups to the class of all groups.

Let $\mathfrak{X}$ be a class of groups and $G$ be a group. A subgroup $H$ of $G$ is said to be a $\mathfrak{X}$-projector of $G$ if $HN/N$ is  $\mathfrak{X}$-maximal in $G$ for every $N\triangleleft G$. Let $\pi$ be a set of primes. Recall that $\mathfrak{G}_\pi$ is the class of all $\pi$-groups. According to  \cite[III, 3.10]{HD}      $\mathfrak{G}_\pi$-projectors exist in every group.

\begin{defn}\label{d1} We shall call a subgroup $H$ of a group $G$  $\sigma^{(1)}$-permutable if it permutes with every $\pi_i$-maximal subgroups of $G$ for all $\pi_i\in\sigma$.\end{defn}

\begin{defn}\label{d2} We shall call a subgroup $H$ of a group $G$  $\sigma^{(2)}$-permutable if it permutes with every $\mathfrak{G}_{\pi_i}$-projector of $G$ for all $\pi_i\in\sigma$.\end{defn}

\begin{defn}\label{d3} We shall call a subgroup $H$ of a group $G$  $\sigma^{(3)}$-permutable if for all $\pi_i\in\sigma$ there is a $\mathfrak{G}_{\pi_i}$-projector $P$ of $G$ such that $HP^x=P^xH$ for all $x\in G$.\end{defn}

Let $\sigma_1=\{\{2\}, \{3\}, \{5\},\dots\}$. Then the concepts of $S$-permutable and $\sigma_1^{(i)}$-permutable subgroups coincides for $i\in\{1, 2, 3\}$.  Let $\pi$ be a set of primes and $G$ be a group. If $H$ is a Hall $\pi$-subgroup of $G$ then it is a $\mathfrak{G}_\pi$-projector of $G$. Therefore every $\sigma$-permutable subgroup  is  $\sigma^{(3)}$-permutable.
Also it is clear that every $\sigma^{(1)}$-permutable subgroup is  $\sigma^{(2)}$-permutable and  every $\sigma^{(2)}$-permutable subgroup is  $\sigma^{(3)}$-permutable. As follows form theorems of Hall  the concepts of $\sigma$-permutable subgroup and $\sigma^{(i)}$-permutable subgroup for $i\in\{1, 2, 3\}$ coincides for a soluble group.  Moreover

\begin{thm}\label{t0} A subgroup $H$ of a group $G$ is $\sigma^{(3)}$-permutable if and only if it  is $\sigma^{(2)}$-permutable. \end{thm}

\begin{conj} A subgroup $H$ of a group $G$ is $\sigma^{(1)}$-permutable if and only if it  is $\sigma^{(2)}$-permutable. \end{conj}


According to \cite{sp4} a group is called $\sigma$-nilpotent if it is the direct product of its Hall $\pi_i$-subgroups for all $\pi_i\in\sigma$. Such classes of groups are the examples of lattice formations (see \cite[Chapter 6]{s9}).  The connection between the class $\mathfrak{N}_\sigma$ of all $\sigma$-nilpotent groups and $\sigma^{(3)}$-permutable subgroups  is shown in

\begin{thm}\label{t1} Let $H$ be a $\sigma^{(3)}$-permutable subgroup of a group   $G$. Then $H^G/H_G$ is $\sigma$-nilpotent.
\end{thm}

\begin{cor}[Kegel \cite{sp1} and Deskins \cite{sp2}]  Let $H$ be a $S$-permutable subgroup of a group $G$ then $H^G/H_G$ is  nilpotent.
\end{cor}

\begin{cor}[{\cite[Theorem B(i)]{sp4}}] Let $G$ be a $E_{\pi_i}$-group for all $\pi_i\in\sigma$. If $H$ is a $\sigma$-permutable subgroup of a group $G$ then $H^G/H_G$ is $\sigma$-nilpotent.
\end{cor}

 Recall \cite{sp4} that  a subgroup $H$ of a group $G$ is called $\sigma$-subnormal in $G$ if there is a subgroup chain $H = H_0\leq H_1\leq\dots\leq  H_n = G$ such that either $H_{j-1}$ is normal in $H_j$ or $\pi(H_j/(H_{j-1})_{H_j})\subseteq\pi_i$ for some $\pi_i\in\sigma$ and $j=1,\dots, n$. In fact, the concept of $\sigma$-subnormal subgroup is equivalent to the concept of $K$-$\mathfrak{N}_\sigma$-subnormal subgroup in the sense  of \cite[6.1.4]{s9}.

 \begin{cor}\label{cor3} Let $H$ be a $\sigma^{(3)}$-permutable subgroup of a group $G$.\!\! Then $H$ is $\sigma$-subnormal in $G$.   \end{cor}

The following theorem shows that conjecture 1 is true for $\sigma$-nilpotent $\sigma^{(2)}$-subnormal subgroups.

\begin{thm}\label{prop1} Let  $H$ be a $\sigma$-nilpotent subgroup of a group   $G$. Then

$(1)$ If $H$ is $\sigma^{(i)}$-permutable in $G$ for some $i\in\{1, 2, 3\}$ then $H$ is $\sigma^{(i)}$-permutable in $G$ for all $i\in\{1, 2, 3\}$.

$(2)$ $H$ is   $\sigma^{(3)}$-permutable in $G$    if and only if every Hall $\pi_i$-subgroup  of $H$ is $\sigma^{(3)}$-permutable in $G$  for all $\pi_i\in\sigma$.
\end{thm}

\begin{cor}[{Schmid \cite{sp3}}] Let $H$ be a nilpotent subgroup of a group $G$. Then $H$   is  $S$-permutable in $G$  if and only if every Sylow subgroup of $H$ is $S$-permutable in $G$.
  \end{cor}

The main result of this paper is
\begin{thm}\label{t3}  The set of all $\sigma^{(3)}$-permutable subgroups of a group $G$  forms a sublattice of the lattice of all   subgroups of $G$.
\end{thm}
\begin{cor}[{\cite[Satz 2]{sp1}}] The set of all $S$-permutable subgroups of a group $G$  forms a sublattice of the lattice of all   subgroups of $G$.\end{cor}
\begin{cor}[{\cite[Theorem C]{sp4}}] Let every subgroup of a group $G$ be a $D_{\pi_i}$-group for all $\pi_i\in\sigma$. Then   the set of all $\sigma$-permutable subgroups of $G$  forms a sublattice of the lattice of all   subgroups of $G$.
 \end{cor}
Our final result concerns the $\sigma^{(3)}$-permutability of the normalizer of a $\sigma^{(3)}$-permutable subgroup.
\begin{thm}\label{t2} If $H$ is a $\sigma^{(3)}$-permutable subgroup of a group $G$ then $N_G(H)$ is also $\sigma^{(3)}$-per\-mutable in $G$.  \end{thm}
\begin{cor}[{Schmid \cite{sp3}}] If $H$ is a $S$-permutable subgroup of   $G$ then   $N_G(H)$ is also  $S$-per\-mutable in   $G$. \end{cor}

\section{Preliminaries}

All unexplained notations and terminologies are standard. The reader is referred to \cite{HD, PFG} if necessary.  Recall that $\mathrm{O}_\pi(G)$ is the unique largest normal $\pi$-subgroup of $G$; $\mathrm{O}^\pi(G)$ is the unique smallest normal subgroup of $G$ for which the corresponding factor group is a $\pi$-group; $H_G$ is the unique largest
normal subgroup of $G$ contained in $H$; $H^G$ is the unique smallest
normal subgroup of $G$ containing $H$.



Let $\mathfrak{F}$ be a homomorph. It is known that if a subgroup $P$ of a group $G$ is an $\mathfrak{F}$-projector then $PN/N$ is an $\mathfrak{F}$-projector of $G/N$. And if $P/N$ is   an $\mathfrak{F}$-projector of $G/N$  then all $\mathfrak{F}$-projectors of $P$ are $\mathfrak{F}$-projectors of $G$ (see \cite[III, 3.7]{HD}). It means that the set $\{PN/N\,|\,P$ is an $\mathfrak{F}$-projector of $G\}$ is the set of all $\mathfrak{F}$-projectors of $G/N$.

\begin{lem}\label{l1} Let  $N$ be a normal subgroup of a group $G$ and $i\in\{2, 3\}$.

$(1)$ If\! $H$\! is a $\sigma^{(i)}$-permutable\! subgroup of   $G$\! then $HN/N$ is a $\sigma^{(i)}$-permutable subgroup of   $G/N$.

$(2)$ If   $H/N$ is\! a $\sigma^{(i)}$-permutable subgroup\! of   $G/N$ then $H$ is a $\sigma^{(i)}$-permutable subgroup of   $G$.
\end{lem}

\begin{proof} Assume that $H$ is $\sigma^{(2)}$-permutable in $G$. Then $HP=PH$ for every $\mathfrak{G}_{\pi_i}$-projector $P$ of $G$ and every $\pi_i\in\sigma$.    So $(HN/N)(PN/N)=HPN/N=(PN/N)(HN/N)$ for every $\mathfrak{G}_{\pi_i}$-projector $P$ of $G$ and every $\pi_i\in\sigma$. It means that $HN/N$ permutes with all   $\mathfrak{G}_{\pi_i}$-projectors  $G/N$ for all $\pi_i\in\sigma$. Thus $HN/N$ is $\sigma^{(2)}$-permutable in $G/N$.

 Assume that $HN/N$ is $\sigma^{(2)}$-permutable in $G/N$. Then $(HN/N)(PN/N)=HPN/N=(PN/N)(HN/N)$ for every $\mathfrak{G}_{\pi_i}$-projector $P$ of $G$ and every $\pi_i\in\sigma$. It means that $(HN)(PN)=HPN=(PN)(HN)$ or $(HN)P=HPN=P(HN)$ for every $\mathfrak{G}_{\pi_i}$-projector $P$ of $G$ and every $\pi_i\in\sigma$. Thus $HN$ is  $\sigma^{(2)}$-permutable in $G$.

 The proof for  $\sigma^{(3)}$-permutable subgroups is analogues. \end{proof}

\begin{lem}[{\cite[Lemma 2.6]{sp4}}]\label{l26} Let $H, K$ be a subgroups of a group  $G$. Then

$(1)$ If $H$ is $\sigma$-subnormal in $G$ then $H\cap K$ is $\sigma$-subnormal in $K$.

$(2)$ If $H$ is $\sigma$-subnormal in $G$ and $|G:H|$ is a $\pi_i$-number for some $\pi_i\in\sigma$ then $\mathrm{O}^{\pi_i}(H)=\mathrm{O}^{\pi_i}(G)$.
       \end{lem}

 \begin{lem}[{\cite[1.2.2]{PFG}}]\label{122}  If a subgroup $H$ of a group $G$ permutes with the subgroups $X$ and $Y$ of $G$ the it is also permutes with their join $\langle X, Y\rangle$.      \end{lem}

 The following lemma directly follows from \cite[6.3.8]{s9}.

 \begin{lem}\label{l6} Let $H$ be a $\sigma$-subnormal $\pi_i$-subgroup for some $\pi_i\in\sigma$. Then $H\leq\mathrm{O}_{\pi_i}(G)$.     \end{lem}

\section{Proofs of the results}

\textbf{Proof of theorem \ref{t1}.} Assume that $(G,H)$ is a counterexample with $|G|+|G:H|$ minimum. From (1) of lemma \ref{l1} and inductive hypothesis it follows that $H_G=1$.

Let $D=\underset{x\in G\setminus N_G(H)}{\cap}\langle H,H^x\rangle$. We have that $H\leq D$ and $\langle H,H^x\rangle$ is $\sigma^{(3)}$-permutable for all $x\in G$ by lemma \ref{122}. It is clear that $D^G=H^G$ and  $D_G=\underset{x\in G\setminus N_G(H)}{\cap}\langle H,H^x\rangle_G$.

Assume that $D=H$. We see that $\langle H,H^x\rangle^G=H^G$ for all $x\in G\setminus N_G(H)$. By induction   $\langle H,H^x\rangle^G/\langle H,H^x\rangle_G=H^G/\langle H,H^x\rangle_G$ is $\sigma$-nilpotent. Since the class of all $\sigma$-nilpotent groups is a formation,  $H^G/\underset{x\in G\setminus N_G(H)}{\cap}\langle H,H^x\rangle_G=H^G/H_G$ is $\sigma$-nilpotent, a contradiction.

 So $H$ is a proper subgroup of $D$.   From $H_G=1$ and $H\neq 1$ it follows that $N_G(H)\neq G$.    So there is a $\pi_i$-element $x\in G\setminus N_G(H)$ for some $\pi_i\in\sigma$. Note that if $N_G(H)$ contains some $\mathfrak{G}_{\pi_i}$-projector of $G$ and all its conjugates in $G$ then it contains all $\pi_i$-elements of $G$.  It means that there is a $\mathfrak{G}_{\pi_i}$-projector  $T$ of $G$ with $HT=TH$  and we may assume that  $x\in T$. Then $H<D\leq\langle H,H^x\rangle\leq\langle H,x\rangle\leq HT=TH$. Hence  $|D:H|$ is a $\pi_i$-number. If there is a $\pi_j$-element $y\in N_G(H)$ for $ i\neq j$ then the same argument shows that $|D:H|$ is a $\pi_j$-number. So $H=D$, a contradiction.

 It means that all $\pi_i'$-elements of $G$ lie in  $N_G(H)$. Hence $O^{\pi_i}(G)\leq N_G(H)$. So $H^G=H^{HT}$. Thus $|H^G:H|$ is a $\pi_i$-number.
From  $H\triangleleft HO^{\pi_i}(G)\leq TO^{\pi_i}(G)=G$ it follows that   $H$ is $\sigma$-subnormal in $G$, and hence in $H^G$ by (1) of lemma \ref{l26}.
 So $O^{\pi_i}(H^G)=O^{\pi_i}(H)$ by (2) of lemma \ref{l26}. Since $H_G=1$,  we see $O^{\pi_i}(H^G)=1$. Thus $H^G$ is a $\pi_i$-group, i.e. $H^G/H_G$ is $\sigma$-nilpotent, the final contradiction. $\square$

\textbf{Proof of corollary 3.} Assume that $H$ is a $\sigma^{(3)}$-permutable subgroup of a group $G$. According to theorem \ref{t1} $H^G/H_G$ is $\sigma$-nilpotent. Now $H/H_G$ is $\sigma$-subnormal in $H^G/H_G\triangleleft G/H_G$. Hence $H/H_G$ is $\sigma$-subnormal in $G/H_G$. Thus $H$ is $\sigma$-subnormal in $G$.  $\square$

\begin{lem}\label{l2} Let $H$ be a $\sigma^{(3)}$-permutable subgroup of a group $G$. Then $\mathrm{O}^{\pi_i}(G)\leq N_G(\mathrm{O}_{\pi_i}(H))$ for every $\pi_i\in\sigma$.        \end{lem}

\begin{proof} Let $\pi_n\in\sigma$ and $P$ be a $\mathfrak{G}_{\pi_i}$-projector of $G$ for some $\pi_i\in\sigma\setminus\{\pi_n\}$ with $P^xH=HP^x$ for all $x\in G$. According to corollary \ref{cor3}  $H$ is $\sigma$-subnormal in $G$. Now $H$ is $\sigma$-subnormal in        $P^xH$ by (1) of lemma \ref{l26}. From (2) of lemma \ref{l26}  it follows that $\mathrm{O}^{\pi_i}(H)=\mathrm{O}^{\pi_i}(P^xH)$. Hence $\mathrm{O}_{\pi_n}(H) char \mathrm{O}^{\pi_i}(H)=\mathrm{O}^{\pi_i}(P^xH)$. Thus $P^x\leq N_G(\mathrm{O}_{\pi_n}(H))$ for all $x\in G$. Since $P$ is  a $\mathfrak{G}_{\pi_i}$-projector of $G$,    $\mathrm{O}^{\pi_i'}(G)\leq N_G(\mathrm{O}_{\pi_n}(H))$. Therefore  $\mathrm{O}^{\pi_n}(G)\leq N_G(\mathrm{O}_{\pi_n}(H))$.  \end{proof}


 \textbf{Proof of theorem \ref{prop1}.}
Since $H$ is $\sigma$-nilpotent, $\mathrm{O}_{\pi_i}(H)$ is the unique Hall $\pi_i$-subgroup of $H$ for all $\pi_i\in\sigma$. Assume that $H$ is  $\sigma^{(3)}$-permutable in $G$. Then $\mathrm{O}^{\pi_i}(G)\leq N_G(\mathrm{O}_{\pi_i}(H))$ by lemma \ref{l2}. Hence $\mathrm{O}_{\pi_i}(H)$ is $\sigma$-subnormal in $G$ and   $\mathrm{O}_{\pi_i}(H)P=P\mathrm{O}_{\pi_i}(H)$ for every $\pi_j$-maximal subgroup $P$ of $G$ and all $\pi_j\in\sigma\setminus\{\pi_i\}$. According to lemma \ref{l6}   $\mathrm{O}_{\pi_i}(H)\leq \mathrm{O}_{\pi_i}(G)$. Therefore $\mathrm{O}_{\pi_i}(H)P=P\mathrm{O}_{\pi_i}(H)=P$ for every $\pi_i$-maximal subgroup $P$ of $G$. Thus $\mathrm{O}_{\pi_i}(H)$ is $\sigma^{(1)}$-permutable (and hence $\sigma^{(3)}$-permutable) in $G$ for all $\pi_i\in\sigma$. Therefore $H$ is $\sigma^{(1)}$-permutable in $G$ by lemma \ref{122}.

Assume that  every Hall $\pi_i$-subgroup  of $H$ is $\sigma^{(3)}$-permutable in $G$  for all $\pi_i\in\sigma$. By (1) every Hall $\pi_i$-subgroup  of $H$ is $\sigma^{(1)}$-permutable in $G$  for all $\pi_i\in\sigma$.  From lemma \ref{122}   it follows that  $H$ is $\sigma^{(1)}$-permutable (and hence $\sigma^{(3)}$-permutable)  in $G$. $\square$

\textbf{Proof of theorem \ref{t0}.}
We need only to prove that every $\sigma^{(3)}$-permutable subgroup is  $\sigma^{(2)}$-permutable. Let $H$ be a   $\sigma^{(3)}$-permutable subgroup  of a group $G$. Then $H/H_G$ is   $\sigma^{(2)}$-permutable in $G/H_G$ by (1) of theorem \ref{prop1}. So $H$ is $\sigma^{(2)}$-permutable in $G$ by (2) of lemma \ref{l1}. $\square$

\textbf{Proof of theorem \ref{t3}.} In fact, in view of lemma \ref{122}, we have only to show that if $A$ and $B$ are $\sigma^{(3)}$-permutable subgroups of $G$, then $C=A\cap B$ is $\sigma^{(3)}$-permutable in $G$. Assume that this statement is false and let a group $G$ be  minimal order counterexample. Then $A_G\cap B_G=1$ by lemma \ref{l1}.  From theorem \ref{t1} it follows that $A^G/A_G$ and $B^G/B_G$ are $\sigma$-nilpotent. Hence $(A^G\cap B^G)/(A_G\cap B^G)$ and $(A^G\cap B^G)/(A^G\cap B_G)$ are $\sigma$-nilpotent. Thus $(A^G\cap B^G)/(A_G\cap B_G)\simeq (A^G\cap B^G)$ is $\sigma$-nilpotent. So $C$ is $\sigma$-nilpotent.

From corollary \ref{cor3} it follows that $A$ and $B$ are $\sigma$-subnormal in $G$.  So $C$ is $\sigma$-subnormal in $G$ by (1) of lemma \ref{l26}. Now every Hall $\pi_i$-subgroup $C_{\pi_i}$ of $C$ is   $\sigma$-subnormal in $G$ for all $\pi_i\in\sigma$. Hence $C_{\pi_i}\leq \mathrm{O}_{\pi_i}(G)$ by lemma \ref{l6} for all $\pi_i\in\sigma$. Thus $C_{\pi_i}$ permutes with every $\mathfrak{G}_{\pi_i}$-projector of $G$ for all $\pi_i\in\sigma$.

Since every $\sigma^{(3)}$-permutable subgroup is a $\sigma^{(2)}$-permutable subgroup by theorem \ref{t0}, $AH\cap BH$ is a subgroup of $G$ for every $\mathfrak{G}_{\pi_i}$-projector $H$ of $G$ and every $\pi_i\in\sigma$. From  (2) of lemma \ref{l26} it follows that $\mathrm{O}^{\pi_i}(A)=\mathrm{O}^{\pi_i}(AH)$ and $\mathrm{O}^{\pi_i}(B)=\mathrm{O}^{\pi_i}(BH)$ for every $\mathfrak{G}_{\pi_i}$-projector $H$ of $G$ and every $\pi_i\in\sigma$. Thus $\mathrm{O}^{\pi_i}(AH\cap BH)\leq C$. Therefore $|(AH\cap BH):C|$ is a $\pi_i$-number.  So $\mathrm{O}^{\pi_i}(C)=\mathrm{O}^{\pi_i}(AH\cap BH)$ by lemma \ref{l26}. Hence $H\leq N_G(C_{\pi_j})$ for every Hall $\pi_j$-subgroup $C_{\pi_j}$ of $C$ for all $\pi_j\in\sigma\setminus\{\pi_i\}$ and   every $\mathfrak{G}_{\pi_i}$-projector $H$ of $G$ and every $\pi_i\in\sigma$. Thus $O^{\pi_i}(G)\leq N_G(C_{\pi_i})$ for every $\pi_i\in\sigma$. It means that $C_{\pi_i}$ permutes with every $\mathfrak{G}_{\pi_j}$-projector for all $\pi_j\in\sigma\setminus\{\pi_i\}$ for every $\pi_i\in\sigma$. Hence $C_{\pi_i}$ is $\sigma^{(3)}$-permutable subgroup  of $G$ for all $\pi_i\in\sigma$. Thus $C$ is  $\sigma^{(3)}$-permutable subgroup of $G$ by theorem \ref{prop1}, the contradiction. $\square$

\textbf{Proof of theorem \ref{t2}.}
Assume that $(G,H)$ is a counterexample with $|G|+|G:H|$ minimum. Applying lemma \ref{l1} we may assume that $H_G=1$. So $H$ is $\sigma$-nilpotent by theorem \ref{t1}. According to (2) of theorem \ref{prop1} every Hall $\pi_i$-subgroup of $H$ is $\sigma^{(3)}$-permutable for all $\pi_i\in \sigma$. Suppose that  every Hall $\pi_i$-subgroup of $H$ is a proper subgroup of $H$  for all $\pi_i\in \sigma$. Therefore $N_G(P)$ is $\sigma^{(3)}$-permutable in $G$ for every Hall $\pi_i$-subgroup $P$ of $H$  for all $\pi_i\in \sigma$    by the choice of $H$ and theorem \ref{prop1}. Now $N_G(H)$ is $\sigma^{(3)}$-permutable in $G$ by theorem \ref{t3}, a contradiction. Thus $H$ is a $\pi_i$-group for some $\pi_i\in \sigma$. Hence $\mathrm{O}^{\pi_i}(G)\leq N_G(H)$ by lemma \ref{l2}. So $N_G(H)P=G$ for every $\mathfrak{G}_{\pi_i}$-projector $P$ of $G$ and $N_G(H)P=N_G(H)$ for every $\mathfrak{G}_{\pi_j}$-projector $P$ of $G$ for all $\pi_j\in\sigma\setminus\{\pi_i\}$. Thus $N_G(H)$ is $\sigma^{(3)}$-permutable in $G$, the final contradiction.     $\square$



\begin{thebibliography}{99}

\leftskip=-7mm
\itemindent=10mm
\parskip=-0mm
\parsep=0mm
\itemsep=0mm
\labelwidth=-5mm

\bibitem{sp1} O.\,H. Kegel,  Sylow-Gruppen und Subnormaheiler
endlicher Gruppen, Math. Z. \textbf{78} (1962), 205--221.

\bibitem{sp2} W.\,E. Deskins, On Quasinormal Subgroups of Finite Groups,
Math. Z. \textbf{82} (1963), 125--132.

\bibitem{sp3} P. Schmid, Subgroups Permutable with All Sylow Subgroups, J. Algebra. \textbf{207} (1998), 285--293.


\bibitem{sp4} A.\,N. Skiba,  On $\sigma$-subnormal and $\sigma$-permutable subgroups of finite groups, J. Algebra. \textbf{436} (2015), 1--16.

\bibitem{HD}
  Doerk, K., Hawkes, T.: Finite soluble groups. / K. Doerk, T. Hawkes. Walter de Gruyter,  1992.

\bibitem{s9}
A. Ballester-Bolinches  and L.\,M. Ezquerro,  Classes of Finite Groups,  Springer, 2006.

\bibitem{PFG}
Ballester-Bolinches, A., Esteban-Romero, R., Asaad, M.: Products of Finite Groups.  Walter de Gruyter, 2010.




\end{thebibliography}
\end{document}